\documentclass[12pt]{article}
\usepackage{amsfonts}
\usepackage{amsmath}
\usepackage{verbatim}
\usepackage{amsmath}
\usepackage{graphicx}
\usepackage[
bookmarks=true,         bookmarksnumbered=true,                        colorlinks=true,
pdfstartview=FitV,
linkcolor=blue, citecolor=blue, urlcolor=blue]{hyperref}

\allowdisplaybreaks[4]
\newtheorem{theorem}{Theorem}[section]

\newtheorem{lemma}[theorem]{Lemma}

\newtheorem{problem}[theorem]{Problem}
\newtheorem{proposition}[theorem]{Proposition}

\newenvironment{proof}[1][Proof]{\noindent \textit{#1.}  }{\ \rule{0.5em}{0.5em} \par }
\newtheorem{prop}[theorem]{Proposition}

\let\Section=\section
\def\section{\setcounter{equation}{0}\Section}

\begin{document}

\title{H\"{o}lder Continuity of the Solution for a Class of Nonlinear SPDE
Arising from One Dimensional Superprocesses}
\author{Yaozhong Hu, \ Fei Lu \ and \ David Nualart\thanks{
David Nualart is supported by the NSF grant DMS0604207. \newline
Keywords: nonlinear stochastic partial differential equation, stochastic
heat kernel, conditional probability density in a random environment,
Malliavin calculus, H\"{o}lder continuity, moment estimates.
} \\
Department of Mathematics \\
University of Kansas \\
Lawrence, Kansas, 66045 USA}
\date{}
\maketitle

\begin{abstract}
The H\"{o}lder continuity of the solution $X_{t}(x)$ to a nonlinear
stochastic partial differential equation (see (\ref{equ1}) below) arising
from one dimensional super process is obtained. It is proved that the H\"{o}%
lder exponent in time variable is as close as to $1/4$, improving the result
of $1/10$ in \cite{LWXZ11}. The method is to use the Malliavin calculus. The
H\"{o}lder continuity in spatial variable $x$ of exponent $1/2$ is also
obtained by using this new approach. This H\"{o}lder continuity result is
sharp since the corresponding linear heat equation has the same H\"{o}lder
continuity.
\end{abstract}

\section{Introduction}

Consider a system of particles indexed by multi-indexes ${\alpha }$ in a
random environment whose motions are described by
\begin{equation}
x_{\alpha }\left( t\right) =x_{\alpha }+B^{\alpha }\left( t\right)
+\int_{0}^{t}\int_{\mathbb{R}}h\left( y-x_{\alpha }\left( u\right) \right)
W\left( du,dy\right) ,~  \label{Particle}
\end{equation}%
where $h\in L^{2}(\mathbb{R})$, $(B^{\alpha }(t);t\geq 0)_{\alpha }$ are
independent Brownian motions and $W$ is a Brownian sheet on $\mathbb{R}%
_{+}\times \mathbb{R}$ independent of $B^{\alpha }$. For more detail about
this model, we refer to Wang (\cite{W97}, \cite{W98}) and Dawson, Li and
Wang \cite{DLW01}. Under some specifications for the branching mechanism and
in the limiting situation, Dawson, Vaillancourt and Wang \cite{DVW00}
obtained that the density of the branching particles satisfies the following
stochastic partial differential equation (SPDE):
\begin{eqnarray}
X_{t}(x) &=&\mu (x)+\int_{0}^{t}\Delta X_{u}(x)dr-\int_{0}^{t}\int_{\mathbb{R%
}}\nabla _{x}\left( h\left( y-x\right) X_{u}\left( x\right) \right) W\left(
du,dy\right)  \notag \\
&&+\int_{0}^{t}\sqrt{X_{u}\left( x\right) }\frac{V\left( du,dx\right) }{dx},
\label{equ1}
\end{eqnarray}%
where $V$ is a Brownian sheet on $\mathbb{R}_{+}\times \mathbb{R}$
independent of $W$. The joint H\"{o}lder continuity of $\left( t,x\right)
\longmapsto X_{t}(x)$ is left as an open problem in \cite{DVW00}.

Let $H_{2}^{k}(\mathbb{R})=\left\{ u\in L^{2}(\mathbb{R});u^{(i)}\in L^{2}(%
\mathbb{R})\text{ for }i=1,2,\dots ,k\right\} $, the Sobolev space with norm
$\left\Vert h\right\Vert _{k,2}^{2}=\sum_{i=0}^{k}\left\Vert
h^{(k)}\right\Vert _{L^{2}(\mathbb{R})}^{2}$. In a recent paper, Li, Wang,
Xiong and Zhou \cite{LWXZ11} proved that $X_{t}(x)$ is almost surely jointly
H\"{o}lder continuous, under the condition that $h\in H_{2}^{2}(\mathbb{R})$
with $\left\Vert h\right\Vert _{1,2}^{2}<2$ and $X_{0}=\mu \in H_{2}^{1}(%
\mathbb{R})$ is bounded. More precisely, they showed that for fixed $t$ its H%
\"{o}lder exponent in $x$ is in $\left( 0,1/2\right) $ and for fixed $x$ its
H\"{o}lder exponent in $t$ is in $\left( 0,1/10\right) $. Comparing to the H%
\"{o}lder continuity for the stochastic heat equation which has the H\"{o}%
lder continuity of $1/4$ in time, it is conjectured that the H\"{o}lder
continuity of $X_{t}(x)$ should also be $1/4$.

The aim of this paper is to provide an affirmative answer to the above
conjecture. Here is the main result of this paper.

\begin{theorem}
\label{main}Suppose that $h\in H_{2}^{2}(\mathbb{R})$ and $X_{0}=\mu \in
L^{2}(\mathbb{R})$ is bounded. Then the solution to $X_{t}\left( x\right) $
is jointly H\"{o}lder continuous with the H\"{o}lder exponent in $x$ in $%
\left( 0,1/2\right) $ and with the H\"{o}lder exponent in $t$ in $\left(
0,1/4\right) $. That is, for any $0\leq s<t\leq T$, $x,y\in \mathbb{R}$ and $%
p\geq 1$, there exists a constant $C$ depending only on $p,T$, $\left\Vert
h\right\Vert _{2,2}$ and $\left\Vert \mu \right\Vert _{L^{2}(\mathbb{R})}$
such that
\begin{equation}
E\left\vert X_{t}\left( y\right) -X_{s}\left( x\right) \right\vert ^{2p}\leq
C(1+t^{-p})(\left\vert x-y\right\vert ^{p-\frac{1}{2}}+(t-s)^{\frac{p}{2}-%
\frac{1}{4}}).  \label{JointH}
\end{equation}
\end{theorem}

Note that the term $t^{-p}$ in the right hand side of (\ref{JointH}) implies
that the H\"{o}lder norm of $X_{t}(x)$ blows up as $t\rightarrow 0$. This
problem arises naturally since we only assume $X_{0}=\mu \in L^{2}(\mathbb{R}%
)$.

When $h=0$ the equation (\ref{equ1}) is reduced to the famous
Dawson-Watanabe equation (process). The study on the joint H\"{o}lder
continuity for this equation has been studied by Konno and Shiga \cite{KS88}
and Reimers \cite{Re89}. The starting point is to interpret the equation
(when $h=0$) in mild form with the heat kernel associated with the Laplacian
$\Delta $ in (\ref{equ1}). Then the properties of the heat kernel (Gaussian
density) can be fully used to analyze the H\"{o}lder continuity.

The straightforward extension of the mild solution concept and technique to
general nonzero $h$ case in (\ref{equ1}) meets a substantial difficulty. To
overcome this difficulty, Li et al \cite{LWXZ11} replace the heat kernel by
a random heat kernel associated with
\begin{equation*}
\int_{0}^{t}\Delta X_{u}(x)dr-\int_{0}^{t}\int_{\mathbb{R}}\nabla
_{x}(h\left( y-x\right) X_{u}(x))W\left( du,dy\right) .
\end{equation*}%
The random heat kernel is given by the conditional transition function of a
typical particle in the system with $W$ given. To be more precise, consider
the spatial motion of a typical particle in the system:%
\begin{equation}
\xi _{t}=\xi _{0}+B_{t}+\int_{0}^{t}\int_{\mathbb{R}}h\left( y-\xi
_{u}\right) W\left( du,dy\right) ,  \label{xi def}
\end{equation}%
where $(B_{t};t\geq 0)$ is a Brownian motion. For $r\leq t$ and $x\in
\mathbb{R}$, define the conditional (conditioned by $W$) transition
probability by
\begin{equation}
P_{t}^{r,x,W}\left( \cdot \right) \equiv P^{W}\left( \xi _{t}\in \cdot |\xi
_{r}=x\right) .  \label{p def}
\end{equation}%
Denote by $p^{W}\left( r,x;t,y\right) $ the density of $P_{t}^{r,x,W}\left(
\cdot \right) $. It is proved that $X_{t}(y)$ has the following convolution
representation:
\begin{eqnarray}
X_{t}(y) &=&\int_{\mathbb{R}}\mu (z)p^{W}\left( 0,z;t,y\right)
dz+\int_{0}^{t}\int_{\mathbb{R}}p^{W}\left( r,z;t,y\right) Z\left(
dr,dz\right)  \label{cnvl rep} \\
&\equiv &X_{t,1}(y)+X_{t,2}(y),  \notag
\end{eqnarray}%
where $Z\left( dr,dz\right) =\sqrt{X_{r}\left( z\right) }V\left(
dr,dz\right) $. Then they introduce a fractional integration by parts
technique to obtain the H\"{o}lder continuity estimates, using Krylov's $%
L_{p}$ theory (cf. Krylov \cite{Kr99}) for linear SPDE.

In this paper, we shall use the techniques from Malliavin calculus to obtain
more precise estimates for the conditional transition function $p^{W}\left(
r,x;t,y\right) $. This allows us to improve the H\"{o}lder continuity in the
time variable for the solution $X_{t}(x)$.

The rest of the paper is organized as follows: In Section 2, we briefly
recall some notations and results on Malliavin calculus. Then we derive
moment estimates for the conditional transition function in Section 3. We
study the H\"{o}lder continuity in spatial and time variables of $X_{t}(x)$
in Section 4 and Section 5 respectively. The proof of Theorem \ref{main} is
concluded in Section 5.

Along the paper, we shall use the following notations; $\left\Vert \cdot
\right\Vert _{H}$ denotes the norm on Hilbert space $H=L^{2}\left( \left[ 0,T%
\right] \right) $, \ $\left\Vert \cdot \right\Vert $ (and $\left\Vert \cdot
\right\Vert _{p}$) denotes the norm on $L^{2}\left( \mathbb{R}\right) $ (and
on $L^{p}\left( \Omega \right) $). The expectation on $\left( \Omega ,%
\mathcal{F},P\right) $ is denoted by $E$ and the conditional expectation
with respect to the process $W$ is denoted by $E^{B}$.

We denote by $C$ a generic positive constant depending only on $p,$ $T$, $%
\left\Vert h\right\Vert _{2,2}$ and $\left\Vert \mu \right\Vert
_{L^{2}\left( \mathbb{R}\right) }$.

\setcounter{equation}{0}

\section{Preliminaries}

Fix a time interval $[0,T]$. Let $(B_{t};t\geq 0)$ be a standard Brownian
motion. Let $\mathcal{S}$ denote the class of smooth random variables of the
form $F=f(B_{t_{1}},...,B_{t_{n}})$, where $t_{1},...,t_{n}$ $\in \lbrack
0,T]$, $n\geq 1$, and $f$ $\in $ $C_{p}^{\infty }(\mathbb{R}^{n})$, the set
of smooth functions $f$ such that $f$ itself and all its partial derivatives
have at most polynomial growth. Given $F=f(B_{t_{1}},...,B_{t_{n}})$ in $%
\mathcal{S}$, its Malliavin derivative $DF$ is the $H$--valued ($%
H=L^{2}\left( \left[ 0,T\right] \right) $) random variable given by
\begin{equation*}
D_{t}F=\sum_{i=1}^{n}\frac{\partial f}{\partial x_{i}}%
(B_{t_{1}},...,B_{t_{n}}))\mathbf{1}_{[0,t_{i}]}(t).
\end{equation*}%
The derivative operator $D$ is a closable and unbounded operator on $%
L^{2}(\Omega )$ taking values in $L^{2}(\Omega ,H)\,$. For any $p\geq 1$, we
denote by $\mathbb{D}^{1,p}$ the closure of $\mathcal{S}$ with respect to
the norm $\left\Vert \cdot \right\Vert _{1,p}$ given by:
\begin{equation*}
\left\Vert DF\right\Vert _{1,p}^{p}=E(\left\vert F\right\vert
^{p})+E(\left\vert \left\vert DF\right\vert \right\vert _{H}^{p}).
\end{equation*}%
We denote by $\delta $ the adjoint operator of $D$, which is an unbounded
from a domain of $L^{2}(\Omega ,H)$ to $L^{2}(\Omega )$. In particular, if $%
u\in Dom(\delta )$, then $\delta (u)$ is characterized by the following
duality relation:
\begin{equation*}
E(\delta (u)F)=E(\langle DF,u\rangle _{H})\quad
\hbox{for any
$F\in \mathbb{D}^{1,2}$}.
\end{equation*}%
The operator $\delta $ is called the $divergence$ operator. The following
two lemmas are from \cite{Nu06}, Propositions 1.5.4 and 2.1.1 and are used
frequently in this paper.

\begin{lemma}
\label{delta norm} The divergence operator $\delta $ is continuous from $%
\mathbb{D}^{1,p}\left( H\right) $ to $L^{p}\left( \Omega \right) ,$ for any $%
p>1$. That is, there exists a constant $C_{p}$ such that
\begin{equation}
\left\Vert \delta \left( u\right) \right\Vert _{L^{p}\left( \Omega \right)
}\leq C_{p}\left( \left\Vert Eu\right\Vert _{H}+\left\Vert Du\right\Vert
_{L^{p}\left( \Omega ,H\otimes H\right) }\right) .  \label{delta}
\end{equation}
\end{lemma}

\begin{lemma}
\label{density} Let $F$ be a random variable in the space $\mathbb{D}^{1,2}$%
, and suppose that $\frac{DF}{\left\Vert DF\right\Vert _{H}^{2}}$ belongs to
the domain of the operator $\delta $ in $L^{2}\left( \Omega \right) $. Then
the law of $F$ has a continuous and bounded density given by%
\begin{equation*}
p\left( x\right) =E\left[ \mathbf{1}_{\left\{ F>x\right\} }\delta \left(
\frac{DF}{\left\Vert DF\right\Vert _{H}^{2}}\right) \right] .
\end{equation*}
\end{lemma}

From $E\delta (u)=0$ for any $u\in Dom(\delta )$ and the H\"{o}lder
inequality it follows that

\begin{lemma}
\label{dens ineq}Let $F$ be a random variable and let $u$ $\in $ $\mathbb{D}%
^{1,q}\left( H\right) $ with $q>1.$ Then for the conjugate pair $p$ and $q$
(i.e. $\frac{1}{p}+\frac{1}{q}=1$),
\begin{equation}
\left\vert E\left[ \mathbf{1}_{\left\{ F>x\right\} }\delta \left( u\right) %
\right] \right\vert \leq \left( P\left( \left\vert F\right\vert >\left\vert
x\right\vert \right) \right) ^{\frac{1}{p}}\left\Vert \delta \left( u\right)
\right\Vert _{L^{q}\left( \Omega \right) }.  \label{density ineq}
\end{equation}
\end{lemma}

\setcounter{equation}{0}

\section{Moment estimates}

In this section, we derive moment estimates for the derivatives of $\xi _{t}$
and the conditional transition function $p^{W}\left( r,x;t,y\right) $.

Recall that $\xi _{t}=\xi _{t}^{r,x}$ with initial value $\xi _{r}=x$ is
given by
\begin{equation}
\xi _{t}=x+B_{r}^{t}+I_{r}^{t}\left( h\right) \,,\ \ 0\leq r<t\leq T\,,
\label{xi0}
\end{equation}%
\ \ where we introduced the notations
\begin{equation}
B_{r}^{t}\equiv B_{t}-B_{r},\text{ and }I_{r}^{t}\left( h\right) \equiv
\int_{r}^{t}\int_{\mathbb{R}}h\left( y-\xi _{u}\right) W\left( du,dy\right) .
\label{notion1}
\end{equation}%
Since $h\in H_{2}^{2}(\mathbb{R})$, by using the standard Picard iteration
scheme, we can prove that such a solution $\xi _{t}$ to the stochastic
differential equation (\ref{xi0}) exists, and by a regularization argument
of $h$ we can prove that $\xi _{t}\in \mathbb{D}^{2,2}$ (here the Malliavin
derivative is with respect to $B$). Taking the Malliavin derivative $%
D_{\theta }$ with respect to $B$, we have%
\begin{equation}
D_{\theta }\xi _{t}=\mathbf{1}_{\left[ r,t\right] }\left( \theta \right) %
\left[ 1-\int_{\theta }^{t}\int_{\mathbb{R}}h^{\prime }\left( y-\xi
_{u}\right) D_{\theta }\xi _{u}W\left( du,dy\right) \right] \,.
\label{Drvt1}
\end{equation}%
Note that
\begin{equation*}
M_{\theta ,t}:=\int_{\theta }^{t}\int_{\mathbb{R}}h^{\prime }\left( y-\xi
_{u}\right) W\left( du,dy\right)
\end{equation*}%
is a martingale with quadratic variation $\langle M\rangle _{\theta
,t}=\left\Vert h^{\prime }\right\Vert ^{2}\left( t-\theta \right) $ for $%
t>\theta $. Thus
\begin{equation}
D_{\theta }\xi _{t}=\mathbf{1}_{\left[ r,t\right] }\left( \theta \right)
\exp \left( M_{\theta ,t}-\frac{1}{2}\left\Vert h^{\prime }\right\Vert
^{2}\left( t-\theta \right) \right) .  \label{D1}
\end{equation}%
As a result, we have
\begin{equation}
D_{\eta }D_{\theta }\xi _{t}=\mathbf{1}_{\left[ r,t\right] }\left( \theta
\right) \exp \left( M_{\theta ,t}-\frac{1}{2}\left\Vert h^{\prime
}\right\Vert ^{2}\left( t-\theta \right) \right) D_{\eta }M_{\theta
,t}=D_{\theta }\xi _{t}\cdot D_{\eta }M_{\theta ,t},  \label{D2}
\end{equation}%
where $D_{\eta }M_{\theta ,t}=\mathbf{1}_{\left[ \theta ,t\right] }\left(
\eta \right) \int_{\eta }^{t}\int_{\mathbb{R}}h^{\prime \prime }\left( y-\xi
_{u}\right) D_{\eta }\xi _{u}W\left( du,dy\right) $.

The next lemma gives estimates for the moments of $D\xi _{t}$ and $D^{2}\xi
_{t}.$

\begin{lemma}
\label{Dmoments}For any $0\leq r<t\leq T$ and $p\geq 1$, we have
\begin{equation}
\left\Vert \left\Vert D\xi _{t}\right\Vert _{H}\right\Vert _{2p}\leq \exp
\left( \left( 2p-1\right) \left\Vert h^{\prime }\right\Vert ^{2}\left(
t-r\right) \right) \left( t-r\right) ^{\frac{1}{2}},  \label{D1p}
\end{equation}%
\begin{equation}
\left\Vert \left\Vert D^{2}\xi _{t}\right\Vert _{H\otimes H}\right\Vert
_{2p}\leq C_{p}\left\Vert h^{\prime \prime }\right\Vert \exp \left( \left(
4p-1\right) \left\Vert h^{\prime }\right\Vert ^{2}\left( t-r\right) \right)
\left( t-r\right) ^{\frac{3}{2}},  \label{D2p}
\end{equation}%
and for any $\gamma >0,$%
\begin{equation}
E(\left\Vert D\xi _{t}\right\Vert _{H}^{-2\gamma })\leq \exp \left( \left(
2\gamma ^{2}+\gamma \right) \left\Vert h^{\prime }\right\Vert ^{2}\left(
t-r\right) \right) \left( t-r\right) ^{-\gamma }.  \label{D1-}
\end{equation}
\end{lemma}

\begin{proof}
Note that for any $p\geq 1$ and $r\leq \theta <t$,
\begin{eqnarray}
\left\Vert D_{\theta }\xi _{t}\right\Vert _{2p}^{2} &=&\left( E\exp \left[
2p\left( M_{\theta ,t}-\frac{1}{2}\left\Vert h^{\prime }\right\Vert
^{2}\left( t-\theta \right) \right) \right] \right) ^{\frac{1}{p}}  \notag \\
&=&\exp \left( \left( 2p-1\right) \left\Vert h^{\prime }\right\Vert
^{2}\left( t-\theta \right) \right) .  \label{D1pE}
\end{eqnarray}%
Then (\ref{D1p}) follows from Minkowski's inequality and (\ref{D1pE}) since%
\begin{equation*}
\left\Vert \left\Vert D\xi _{t}\right\Vert _{H}\right\Vert _{2p}^{2}=\left[
E\left( \int_{r}^{t}\left\vert D_{\theta }\xi _{t}\right\vert ^{2}d\theta
\right) ^{p}\right] ^{\frac{1}{p}}\leq \int_{r}^{t}\left\Vert D_{\theta }\xi
_{t}\right\Vert _{2p}^{2}d\theta .
\end{equation*}%
Applying the Burkholder-Davis-Gundy inequality we have for any $r\leq \theta
\leq \eta <t$%
\begin{eqnarray}
\left\Vert D_{\eta }M_{\theta ,t}\right\Vert _{2p}^{2} &\leq &C_{p}\left(
E\left\vert \int_{\eta }^{t}\int_{\mathbb{R}}\left\vert h^{\prime \prime
}\left( y-\xi _{u}\right) D_{\eta }\xi _{u}\right\vert ^{2}dydu\right\vert
^{p}\right) ^{\frac{1}{p}}  \notag \\
&\leq &C_{p}\left\Vert h^{\prime \prime }\right\Vert ^{2}\int_{\eta
}^{t}\left\Vert D_{\eta }\xi _{u}\right\Vert _{2p}^{2}du.  \label{DMp}
\end{eqnarray}%
Combining (\ref{D2}), (\ref{D1pE}) and (\ref{DMp}) yields for any $r\leq
\theta \leq \eta <t$
\begin{eqnarray}
\left\Vert D_{\eta }D_{\theta }\xi _{t}\right\Vert _{2p}^{2} &=&\left\Vert
D_{\theta }\xi _{t}D_{\eta }M_{\theta ,t}\right\Vert _{2p}^{2}\leq
\left\Vert D_{\theta }\xi _{t}\right\Vert _{4p}^{2}\left\Vert D_{\eta
}M_{\theta ,t}\right\Vert _{4p}^{2}  \notag \\
&\leq &C_{p}\left\Vert h^{\prime \prime }\right\Vert ^{2}\exp \left( 2\left(
4p-1\right) \left\Vert h^{\prime }\right\Vert ^{2}\left( t-\theta \right)
\right) \left( t-\eta \right) .  \label{D2pE}
\end{eqnarray}%
An application of Minkowski's inequality implies that
\begin{equation*}
\left\Vert \left\Vert D^{2}\xi _{t}\right\Vert _{H\otimes H}\right\Vert
_{2p}^{2}\leq \int_{r}^{t}\int_{r}^{t}\left\Vert D_{\eta }D_{\theta }\xi
_{t}\right\Vert _{2p}^{2}d\theta d\eta .
\end{equation*}%
This yields (\ref{D2p}).

For the negative moments of $\left\Vert D\xi _{t}\right\Vert _{H},$ by
Jensen's inequality we have%
\begin{equation*}
E\left( \left\Vert D\xi _{t}\right\Vert _{H}^{-2\gamma }\right) =E\left(
\int_{r}^{t}\left\vert D_{\theta }\xi _{t}\right\vert ^{2}d\theta \right)
^{-\gamma }\leq \left( t-r\right) ^{-\gamma -1}\int_{r}^{t}E\left\vert
D_{\theta }\xi _{t}\right\vert ^{-2\gamma }d\theta .
\end{equation*}%
Then, $\left( \ref{D1-}\right) $ follows immediately.
\end{proof}

The moment estimates of the Malliavin derivatives of the difference $\xi
_{t}-\xi _{s}$ can also be obtained in a similar way. The next lemma gives
these estimates.

\begin{lemma}
\label{D12diff}For $0\leq s<t\leq T$ and $p\geq 1,$ we have%
\begin{equation}
\left\Vert \left\Vert D\left( \xi _{t}-\xi _{s}\right) \right\Vert
_{H}\right\Vert _{2p}<C\left( t-s\right) ^{\frac{1}{2}},  \label{D1diff}
\end{equation}%
and%
\begin{equation}
\left\Vert \left\Vert D^{2}\left( \xi _{t}-\xi _{s}\right) \right\Vert
_{H\otimes H}\right\Vert _{2p}<C\left( t-s\right) ^{\frac{3}{2}}.
\label{D2diff}
\end{equation}
\end{lemma}

\begin{proof}
Similar to (\ref{Drvt1}), we have
\begin{eqnarray}
D_{\theta }\xi _{t} &=&D_{\theta }\xi _{s}+\mathbf{1}_{[s,t]}(\theta
)-\int_{\theta \vee s}^{t}\int_{\mathbb{R}}h^{\prime }\left( y-\xi
_{u}\right) D_{\theta }\xi _{u}W\left( du,dy\right)  \notag \\
&=&D_{\theta }\xi _{s}+\mathbf{1}_{[s,t]}(\theta )-I_{\theta }^{t}\left(
h^{\prime }D_{\theta }\xi _{.}\right) ,  \label{D1diff1}
\end{eqnarray}%
where henceforth for any process $Y=(Y_{t},0\leq t\leq T)$ and $f\in L^{2}(%
\mathbb{R})$, we denote
\begin{equation*}
I_{\theta }^{t}\left( fY_{\cdot }\right) =\mathbf{1}_{[s,t]}(\theta
)\int_{\theta }^{t}\int_{\mathbb{R}}f\left( y-\xi _{u}\right) Y_{u}W\left(
du,dy\right) .
\end{equation*}%
Applying the Burkholder-Davis-Gundy inequality with (\ref{D1pE}), we obtain
for $s\leq \theta \leq t$%
\begin{eqnarray}
&&\left\Vert I_{\theta }^{t}\left( h^{\prime }D_{\theta }\xi _{.}\right)
\right\Vert _{2p}^{2}\leq \left( E\left\vert \int_{\theta }^{t}\int_{\mathbb{%
R}}\left\vert h^{\prime }\left( y-\xi _{u}\right) D_{\theta }\xi
_{u}\right\vert ^{2}dudy\right\vert ^{p}\right) ^{\frac{1}{p}}  \notag \\
&\leq &\left\Vert h^{\prime }\right\Vert ^{2}\exp \left( \left( 2p-1\right)
\left\Vert h^{\prime }\right\Vert ^{2}\left( t-\theta \right) \right) \left(
t-\theta \right) .  \label{D1diff2}
\end{eqnarray}%
Then (\ref{D1diff}) follows from (\ref{D1diff1}) and (\ref{D1diff2}) since
\begin{eqnarray*}
&&\left( E\left\Vert D\xi _{t}-D\xi _{s}\right\Vert _{H}^{2p}\right) ^{\frac{%
1}{p}}=\left[ E\left( \int_{0}^{T}\left\vert \mathbf{1}_{[s,t]}(\theta
)+I_{\theta }^{t}\left( h^{\prime }D_{\theta }\xi _{.}\right) \right\vert
^{2}d\theta \right) ^{p}\right] ^{\frac{1}{p}} \\
&\leq &\int_{0}^{T}\left( E\left\vert \mathbf{1}_{[s,t]}(\theta )+I_{\theta
}^{t}\left( h^{\prime }D_{\theta }\xi _{.}\right) \right\vert ^{2p}\right) ^{%
\frac{1}{p}}d\theta \\
&\leq &2\left( t-s\right) +2\int_{s}^{t}\left( E\left\vert I_{\theta
}^{t}\left( h^{\prime }D_{\theta }\xi _{.}\right) \right\vert ^{2p}\right) ^{%
\frac{1}{p}}d\theta \\
&\leq &2\left( 1+\left\Vert h^{\prime }\right\Vert ^{2}\exp \left( \left(
2p-1\right) \left\Vert h^{\prime }\right\Vert ^{2}\left( t-s\right) \right)
\right) \left( t-s\right) .
\end{eqnarray*}

For moments of $D^{2}\left( \xi _{t}-\xi _{s}\right) $, from (\ref{D1diff1})
we have%
\begin{equation}
D_{\eta ,\theta }^{2}\left( \xi _{t}-\xi _{s}\right) =-D_{\eta }I_{\theta
}^{t}\left( h^{\prime }D_{\theta }\xi _{.}\right) =I_{\eta }^{t}\left(
h^{\prime \prime }D_{\theta }\xi .D_{\eta }\xi .\right) -I_{\eta }^{t}\left(
h^{\prime }D_{\eta ,\theta }^{2}\xi .\right) .  \label{D2diff1}
\end{equation}%
In a similar way as above we can get (\ref{D2diff}).
\end{proof}

Next we derive some estimates for the density $p^{W}\left( r,x;t,y\right) $
of the conditional transition probability defined in (\ref{p def}). Denote
\begin{equation}
u_{t}\equiv \frac{D\xi _{t}}{\left\Vert D\xi _{t}\right\Vert _{H}^{2}}.
\label{notion2}
\end{equation}%
The next two lemmas give estimates of the divergence of $u_{t}$ and $%
u_{t}-u_{s}$, which are important to derive the moment estimates of $%
p^{W}\left( r,x;t,y\right) $.

\begin{lemma}
\label{Edelta u}For any $p\geq 1$ and $0\leq r<t\leq T$, we have%
\begin{equation}
\left\Vert \delta \left( u_{t}\right) \right\Vert _{p}\leq C\left(
t-r\right) ^{-\frac{1}{2}}.  \label{E delta u}
\end{equation}
\end{lemma}

\begin{proof}
Using the estimate $\left( \ref{delta}\right) $ we obtain
\begin{eqnarray*}
\left\Vert \delta \left( u_{t}\right) \right\Vert _{p} &=&\left( E\left\vert
\delta \left( u_{t}\right) \right\vert ^{p}\right) ^{\frac{1}{p}}\leq \left[
E\left( E^{B}\left\vert \delta \left( u_{t}\right) \right\vert ^{p}\right) %
\right] ^{\frac{1}{p}} \\
&\leq &C_{p}\left( E\left[ \left\Vert E^{B}u_{t}\right\Vert _{H}^{p}+\left(
E^{B}\left\Vert Du_{t}\right\Vert _{H\otimes H}^{p}\right) \right] \right) ^{%
\frac{1}{p}} \\
&\leq &C_{p}\left( \left\Vert \left\Vert u_{t}\right\Vert _{H}\right\Vert
_{p}+\left\Vert \left\Vert Du_{t}\right\Vert _{H\otimes H}\right\Vert
_{p}\right) .
\end{eqnarray*}%
We have
\begin{equation*}
Du_{t}=\frac{D^{2}\xi _{t}}{\left\Vert D\xi _{t}\right\Vert _{H}^{2}}-2\frac{%
\left\langle D^{2}\xi _{t},D\xi _{t}\otimes D\xi _{t}\right\rangle
_{H\otimes H}}{\left\Vert D\xi _{t}\right\Vert _{H}^{4}},
\end{equation*}%
and consequently $\left\Vert Du_{t}\right\Vert _{H\otimes H}\leq \frac{%
3\left\Vert D^{2}\xi _{t}\right\Vert _{H\otimes H}}{\left\Vert D\xi
_{t}\right\Vert _{H}^{2}}$. Hence, for any positive number $\alpha ,\beta >1$
such that $\frac{1}{a}+\frac{1}{\beta }=\frac{1}{p}$, applying (\ref{D2p})
and (\ref{D1-}) we obtain (\ref{E delta u}):%
\begin{eqnarray*}
\left\Vert \delta \left( u_{t}\right) \right\Vert _{p} &\leq &C_{p}\left(
\left\Vert \left\Vert D\xi _{t}\right\Vert _{H}^{-1}\right\Vert
_{p}+3\left\Vert D^{2}\xi _{t}\right\Vert _{L^{\alpha }\left( \Omega
,H\otimes H\right) }\left\Vert \left\Vert D\xi _{t}\right\Vert
_{H}^{-2}\right\Vert _{\beta }\right) \\
&\leq &C\left( p,\left\Vert h^{\prime }\right\Vert ,\left\Vert h^{\prime
\prime }\right\Vert ,T\right) \left( (t-r)^{-\frac{1}{2}}+(t-r)^{\frac{3}{2}%
}(t-r)^{-1}\right) .
\end{eqnarray*}%
This proves the lemma.
\end{proof}

\begin{lemma}
\label{delta timediff}For $p\geq 1$, and $0\leq r<s<t\leq T,$
\begin{equation}
\left\Vert \delta \left( u_{t}-u_{s}\right) \right\Vert _{2p}\leq C\left(
t-s\right) ^{\frac{1}{2}}\left( s-r\right) ^{-\frac{1}{2}}\left( t-r\right)
^{-\frac{1}{2}}.  \label{tdelta}
\end{equation}
\end{lemma}

\begin{proof}
Using (\ref{D1diff1}) we can write
\begin{equation*}
u_{t}-u_{s}=\frac{D\xi _{t}}{\left\Vert D\xi _{t}\right\Vert _{H}^{2}}-\frac{%
D\xi _{s}}{\left\Vert D\xi _{s}\right\Vert _{H}^{2}}=A_{1}+A_{2}+A_{3},
\end{equation*}%
where
\begin{equation*}
A_{1}=D\xi _{s}\left( \frac{1}{\left\Vert D\xi _{s}\right\Vert _{H}^{2}}-%
\frac{1}{\left\Vert D\xi _{t}\right\Vert _{H}^{2}}\right) ,A_{2}=\frac{%
\mathbf{1}_{[s,t]}(\theta )}{\left\Vert D\xi _{t}\right\Vert _{H}^{2}},A_{3}=%
\frac{I_{\theta }^{t}\left( h^{\prime }D_{\theta }\xi _{.}\right) }{%
\left\Vert D\xi _{t}\right\Vert _{H}^{2}}.
\end{equation*}%
As a consequence, we have%
\begin{equation}
\left\Vert \delta \left( u_{t}-u_{s}\right) \right\Vert _{2p}\leq
\sum_{i=1}^{3}\left\Vert \delta A_{i}\right\Vert _{2p}.  \label{td}
\end{equation}%
For simplicity we introduce the following notation%
\begin{equation*}
V_{t}\equiv \left\Vert D\xi _{t}\right\Vert _{H}\text{, }N_{t}\equiv
\left\Vert D^{2}\xi _{t}\right\Vert _{H\otimes H}\text{, }Y_{i}=\left\Vert
D^{i}\left( \xi _{t}-\xi _{s}\right) \right\Vert _{H^{\otimes i}}\text{, }%
i=1,2.
\end{equation*}%
Note that
\begin{equation*}
\left\Vert A_{1}\right\Vert _{H}=\frac{\langle D\xi _{t}-D\xi _{s},D\xi
_{t}+D\xi _{s}\rangle }{\left\Vert D\xi _{s}\right\Vert _{H}\left\Vert D\xi
_{t}\right\Vert _{H}^{2}}\leq Y_{1}\left(
V_{t}^{-2}+V_{s}^{-1}V_{t}^{-1}\right) ,
\end{equation*}%
and%
\begin{eqnarray*}
&&\left\Vert DA_{1}\right\Vert _{H\otimes H}=\left\Vert D\left( \frac{D\xi
_{s}\langle D\xi _{t}-D\xi _{s},D\xi _{t}+D\xi _{s}\rangle }{\left\Vert D\xi
_{s}\right\Vert _{H}^{2}\left\Vert D\xi _{t}\right\Vert _{H}^{2}}\right)
\right\Vert _{H\otimes H} \\
&\leq &Y_{1}N_{s}\left( V_{s}^{-2}V_{t}^{-1}+V_{s}^{-1}V_{t}^{-2}\right)
+Y_{2}\left( V_{s}^{-1}V_{t}^{-1}+V_{t}^{-2}\right) \\
&&+Y_{1}\left( N_{t}+N_{s}\right) V_{s}^{-1}V_{t}^{-2} \\
&&+2Y_{1}\left[ N_{s}\left( V_{s}^{-2}V_{t}^{-1}+V_{s}^{-1}V_{t}^{-2}\right)
+N_{t}\left( V_{t}^{-3}+V_{s}^{-1}V_{t}^{-2}\right) \right] .
\end{eqnarray*}%
As a consequence, applying Lemma \ref{delta norm} and H\"{o}lder's
inequality we get
\begin{eqnarray*}
&&\left\Vert \delta \left( A_{1}\right) \right\Vert _{2p}\leq C\left(
\left\Vert \left\Vert A_{1}\right\Vert _{H}\right\Vert _{2p}+\left\Vert
\left\Vert DA_{1}\right\Vert _{H\otimes H}\right\Vert _{2p}\right) \\
&\leq &C\left\Vert Y_{1}\right\Vert _{4p}\left( \left\Vert
V_{t}^{-2}\right\Vert _{4p}+\left\Vert V_{t}^{-1}\right\Vert _{8p}\left\Vert
V_{s}^{-1}\right\Vert _{8p}\right) \\
&&+C\left\Vert Y_{1}\right\Vert _{8p}\left\Vert N_{s}\right\Vert _{8p}\left(
\left\Vert V_{t}^{-1}\right\Vert _{8p}\left\Vert V_{s}^{-2}\right\Vert
_{8p}+\left\Vert V_{s}^{-1}\right\Vert _{8p}\left\Vert V_{t}^{-2}\right\Vert
_{8p}\right) \\
&&+C\left\Vert Y_{2}\right\Vert _{4p}\left( \left\Vert V_{t}^{-1}\right\Vert
_{8p}\left\Vert V_{s}^{-1}\right\Vert _{8p}+\left\Vert V_{t}^{-2}\right\Vert
_{4p}\right) \\
&&+C\left\Vert Y_{1}\right\Vert _{8p}\left( \left\Vert N_{s}\right\Vert
_{8p}+\left\Vert N_{t}\right\Vert _{8p}\right) \left\Vert
V_{t}^{-2}\right\Vert _{8p}\left\Vert V_{s}^{-1}\right\Vert _{8p} \\
&&+2C\left\Vert Y_{1}\right\Vert _{8p}\left\Vert N_{s}\right\Vert
_{8p}\left( \left\Vert V_{t}^{-1}\right\Vert _{8p}\left\Vert
V_{s}^{-2}\right\Vert _{8p}+\left\Vert V_{t}^{-2}\right\Vert _{8p}\left\Vert
V_{s}^{-1}\right\Vert _{8p}\right) \\
&&+2C\left\Vert Y_{1}\right\Vert _{8p}\left\Vert N_{t}\right\Vert
_{8p}\left( \left\Vert V_{t}^{-3}\right\Vert _{4p}+\left\Vert
V_{t}^{-2}\right\Vert _{8p}\left\Vert V_{s}^{-1}\right\Vert _{8p}\right) .
\end{eqnarray*}%
From Lemma \ref{Dmoments} and Lemma \ref{D12diff} it follows that
\begin{equation}
\left\Vert \delta \left( A_{1}\right) \right\Vert _{2p}\leq C\left(
t-s\right) ^{\frac{1}{2}}\left( s-r\right) ^{-\frac{1}{2}}\left( t-r\right)
^{-\frac{1}{2}}.  \label{tdA1}
\end{equation}%
Note that $\left\Vert A_{2}\right\Vert _{H}=\left\Vert \frac{\mathbf{1}%
_{[s,t]}(\theta )}{\left\Vert D\xi _{t}\right\Vert _{H}^{2}}\right\Vert
_{H}=\left\Vert D\xi _{t}\right\Vert _{H}^{-2}\left( t-s\right) ^{\frac{1}{2}%
}$ and%
\begin{equation*}
\left\Vert DA_{2}\right\Vert _{H\otimes H}\leq 2\left\Vert D\xi
_{t}\right\Vert _{H}^{-3}\left\Vert D^{2}\xi _{t}\right\Vert _{H\otimes
H}\left( t-s\right) ^{\frac{1}{2}}.
\end{equation*}%
Then, by Lemma \ref{delta norm}, H\"{o}lder's inequality and Lemma \ref%
{Dmoments} we see that
\begin{eqnarray}
&&\left\Vert \delta \left( A_{2}\right) \right\Vert _{2p}\leq C\left(
\left\Vert \left\Vert A_{2}\right\Vert _{H}\right\Vert _{2p}+\left\Vert
DA_{2}\right\Vert _{2p}\right)  \notag \\
&\leq &C\left( t-s\right) ^{\frac{1}{2}}\left( \left\Vert
V_{t}^{-2}\right\Vert _{2p}+\left\Vert D^{2}\xi _{t}\right\Vert
_{4p}\left\Vert V_{t}^{-1}\right\Vert _{4p}\right)  \notag \\
&\leq &2C\left( t-s\right) ^{\frac{1}{2}}\left( \left( t-r\right)
^{-1}+1\right) .  \label{tdA2}
\end{eqnarray}

For the term $A_{3}$, we apply Minkowski's inequality and the
Burkholder-Davis-Gundy inequality and use (\ref{D1diff2}). Thus for any $%
p\geq 1,$
\begin{eqnarray}
&&\left\Vert \left\vert \left\vert I_{\theta }^{t}\left( h^{\prime
}D_{\theta }\xi _{.}\right) \right\vert \right\vert _{H}\right\Vert
_{2p}=\left( E\left\vert \int_{s}^{t}I_{\theta }^{t}(h^{\prime }D_{\theta
}\xi _{.})^{2}d\theta \right\vert ^{p}\right) ^{\frac{1}{2p}}  \notag \\
&\leq &C_{p}\left( \int_{s}^{t}\left\Vert I_{\theta }^{t}\left( h^{\prime
}D\xi .\right) \right\Vert _{2p}^{2}d\theta \right) ^{\frac{1}{2}}  \notag \\
&\leq &C_{p}\left\Vert h^{\prime }\right\Vert \exp \left( \left( 2p-1\right)
\left\Vert h^{\prime }\right\Vert ^{2}\left( t-r\right) \right) \left(
t-s\right) ^{\frac{1}{2}}.  \label{tdA31}
\end{eqnarray}%
From (\ref{D2diff1}) it follows that
\begin{eqnarray*}
\left\Vert DA_{3}\right\Vert _{H\otimes H} &\leq &\left\Vert D^{2}\left( \xi
_{t}-\xi _{s}\right) \right\Vert _{H\otimes H}\left\Vert D\xi
_{t}\right\Vert _{H}^{-2} \\
&&+2\left\vert \left\vert I_{\theta }^{t}\left( h^{\prime }D\xi .\right)
\right\vert \right\vert _{H}\left\Vert D^{2}\xi _{t}\right\Vert _{H\otimes
H}\left\Vert D\xi _{t}\right\Vert _{H}^{-3}.
\end{eqnarray*}%
Combining this with Lemma \ref{delta norm}, H\"{o}lder's inequality, Lemma %
\ref{D12diff} and (\ref{tdA31}) we deduce
\begin{eqnarray}
&&\left\Vert \delta (A_{3})\right\Vert _{2p}\leq C_{p}\left( \left\Vert
\left\Vert A_{3}\right\Vert _{H}\right\Vert _{2p}+\left\Vert
DA_{3}\right\Vert _{2p}\right)  \notag \\
&\leq &C_{p}\left\Vert \left\vert \left\vert I_{s}^{t}\left( h^{\prime
}D_{\theta }\xi _{.}\right) \right\vert \right\vert _{H}\right\Vert
_{2p}\left\Vert V_{t}^{-2}\right\Vert _{2p}+\left\Vert Y_{2}\right\Vert
_{4p}\left\Vert V_{t}^{-2}\right\Vert _{4p}  \notag \\
&&+2C_{p}\left\Vert \left\vert \left\vert I_{s}^{t}\left( h^{\prime
}D_{\theta }\xi _{.}\right) \right\vert \right\vert _{H}\right\Vert
_{4p}\left\Vert V_{t}^{-3}\right\Vert _{8p}\left\Vert N_{t}\right\Vert _{8p}
\notag \\
&\leq &C\left( t-s\right) ^{\frac{1}{2}}\left( t-r\right) ^{-1}.
\label{tdA3}
\end{eqnarray}%
Substituting (\ref{tdA1}), (\ref{tdA2}) and (\ref{tdA3}) into (\ref{td})
yields (\ref{tdelta}).
\end{proof}

Now we provide the moment estimates for the conditional transition
probability density $p^{W}\left( r,x;t,y\right)$.

\begin{lemma}
\label{dens bd}Let $c=1\vee {\Vert h\Vert }^{2}$. For any $0\leq r<t\leq T$,
$y\in \mathbb{R}$ and $p\geq 1$,
\begin{equation}
\left( E\left\vert p^{W}\left( r,x;t,y\right) \right\vert ^{2p}\right) ^{%
\frac{1}{2p}}\leq 2\exp \left( -\frac{\left( x-y\right) ^{2}}{64pc\left(
t-r\right) }\right) \left\Vert \delta \left( u_{t}\right) \right\Vert _{4p}.
\label{dens ctrl}
\end{equation}
\end{lemma}

\begin{proof}
By Lemma \ref{density} we can write
\begin{equation}
p^{W}\left( r,x;t,y\right) =E^{B}\left( \mathbf{1}_{\left\{ \xi
_{t}>y\right\} }\delta \left( u_{t}\right) \right) =E^{B}[\mathbf{1}%
_{\left\{ B_{r}^{t}+I_{r}^{t}(h)>y-x\right\} }\delta \left( u_{t}\right) ],
\label{dens}
\end{equation}%
where $B_{r}^{t}$ and $I_{r}^{t}(h)$ are defined in (\ref{notion1}). Then, (%
\ref{density ineq}) implies
\begin{eqnarray}
&&\left( E\left\vert p^{W}\left( r,x;t,y\right) \right\vert ^{2p}\right) ^{%
\frac{1}{2p}}  \notag \\
&\leq &\left( E\left[ \left( P^{B}\left( \left\vert
B_{r}^{t}+I_{r}^{t}\left( h\right) \right\vert >\left\vert y-x\right\vert
\right) \right) ^{p}\left( E^{B}\left\vert \delta \left( u_{t}\right)
\right\vert ^{2}\right) ^{p}\right] \right) ^{\frac{1}{2p}}  \notag \\
&\leq &\left\Vert \delta \left( u_{t}\right) \right\Vert _{4p}\left( E\left(
P^{B}\left( \left\vert B_{r}^{t}+I_{r}^{t}\left( h\right) \right\vert
>\left\vert y-x\right\vert \right) \right) ^{2p}\right) ^{\frac{1}{4p}}.
\label{dens2}
\end{eqnarray}%
Applying Chebyshev and Jensen's inequalities, we\ have for $p\geq 1$,
\begin{eqnarray}
&&E\left\vert P^{B}\left( \left\vert B_{r}^{t}+I_{r}^{t}\left( h\right)
\right\vert >\left\vert y-x\right\vert \right) \right\vert ^{2p}  \notag \\
&\leq &\exp \left( \frac{-2p\left( x-y\right) ^{2}}{32pc\left( t-r\right) }%
\right) E\left\vert E^{B}\exp \frac{\left( B_{r}^{t}+I_{r}^{t}(h)\right) ^{2}%
}{32pc\left( t-r\right) }\right\vert ^{2p}  \notag \\
&\leq &\exp \left( \frac{-\left( x-y\right) ^{2}}{16c\left( t-r\right) }%
\right) E\exp \frac{\left( B_{r}^{t}+I_{r}^{t}(h)\right) ^{2}}{16c\left(
t-r\right) }.  \label{Ulamda}
\end{eqnarray}%
Using the fact that for $0\leq \nu <1/8$ and Gaussian random variables $X,Y$%
,
\begin{equation*}
Ee^{\nu \left( X+Y\right) ^{2}}\leq Ee^{2\nu \left( X^{2}+Y^{2}\right) }\leq
\left( Ee^{4\nu X^{2}}\right) ^{\frac{1}{2}}\left( Ee^{4\nu Y^{2}}\right) ^{%
\frac{1}{2}}=\left( 1-8\nu \right) ^{-\frac{1}{2}},
\end{equation*}%
and noticing that $B_{r}^{t}$ and $I_{r}^{t}(h)$ are Gaussian, we have
\begin{equation}
E\exp \frac{\left( B_{r}^{t}+I_{r}^{t}(h)\right) ^{2}}{16c\left( t-r\right) }%
\leq \left( 1-\frac{1}{2c}\right) ^{-\frac{1}{2}}\leq \sqrt{2}.
\label{Gaussian}
\end{equation}

Combining (\ref{dens2})--(\ref{Gaussian}), we get (\ref{dens ctrl}).
\end{proof}

\setcounter{equation}{0}

\section{H\"older continuity in spatial variable}

In this section, we obtain the H\"{o}lder continuity of $X_{t}(y)$ with
respect to $y$. More precisely, we show that for $t>0$ fixed, $X_{t}(y)$ is
almost surely H\"{o}lder continuous in $y$ with any exponent in $\left(
0,1/2\right) $. This result was proved in \cite{LWXZ11}. Here we provide a
different proof based on Malliavin calculus. We continue to use the
notations $B_{r}^{t}$, $I_{r}^{t}(h)$ (defined by (\ref{notion1})) and $%
u_{t} $ (defined by (\ref{notion2})).

\begin{prop}
\label{Space}Suppose that $h\in H_{2}^{2}(\mathbb{R})$ and $X_{0}=\mu \in
L^{2}(\mathbb{R})$ is bounded. Then, for any $t\in (0,T]$, $\alpha \in
\left( 0,1\right) $ and $p>1$, there exists a constant $C$ depending only on
$p,T$, $\left\Vert h\right\Vert _{2,2}$ and $\left\Vert \mu \right\Vert
_{L^{2}(\mathbb{R})}$ such that
\begin{equation}
E\left\vert X_{t}\left( y_{2}\right) -X_{t}\left( y_{1}\right) \right\vert
^{2p}\leq C(1+t^{-p})\left( y_{2}-y_{1}\right) ^{\alpha p}.  \label{space}
\end{equation}
\end{prop}

\begin{proof}
We will use the convolution representation (\ref{cnvl rep}), where the two
terms $X_{t,1}\left( y\right) $ and $X_{t,2}\left( y\right) $ will be
estimated separately.

We start with $X_{t,2}\left( y\right) $. Suppose $y_{1}<y_{2}\in \mathbb{R}$%
. Note that $\mathbf{1}_{\left\{ \xi _{t}>y_{1}\right\} }-\mathbf{1}%
_{\left\{ \xi _{t}>y_{2}\right\} }=\mathbf{1}_{\left\{ y_{1}<\xi _{t}\leq
y_{2}\right\} }$ and
\begin{equation*}
E^{B}\mathbf{1}_{\left\{ y_{1}<\xi _{t}\leq y_{2}\right\} }=P^{B}\left\{
y_{1}<\xi _{t}\leq y_{2}\right\} =\int_{y_{1}}^{y_{2}}p^{W}\left(
r,x;t,z\right) dz.
\end{equation*}%
Therefore by (\ref{dens}) we have
\begin{eqnarray*}
&&\left\vert p^{W}\left( r,x;t,y_{1}\right) -p^{W}\left( r,x;t,y_{2}\right)
\right\vert ^{2}=\left\vert E^{B}\left[ \mathbf{1}_{\left\{ y_{1}<\xi
_{t}<y_{2}\right\} }\delta \left( u_{t}\right) \right] \right\vert ^{2} \\
&\leq &E^{B}\left\vert \delta \left( u_{t}\right) \right\vert
^{2}\int_{y_{1}}^{y_{2}}p^{W}\left( r,x;t,z\right) dz.
\end{eqnarray*}%
Hence,
\begin{eqnarray}
&&\left( E\left\vert p^{W}\left( r,x;t,y_{1}\right) -p^{W}\left(
r,x;t,y_{2}\right) \right\vert ^{2\left( 2p-1\right) }\right) ^{\frac{1}{2p-1%
}}  \notag \\
&\leq &\left\Vert \delta \left( u_{t}\right) \right\Vert
_{4(2p-1)}^{2}\int_{y_{1}}^{y_{2}}\left\Vert p^{W}\left( r,x;t,z\right)
\right\Vert _{2(2p-1)}dz.  \label{space4}
\end{eqnarray}%
Lemma \ref{Edelta u} and Lemma \ref{dens bd} yield
\begin{eqnarray}
&&\int_{\mathbb{R}}\left( E\left\vert p^{W}\left( r,x;t,y_{1}\right)
-p^{W}\left( r,x;t,y_{2}\right) \right\vert ^{2\left( 2p-1\right) }\right) ^{%
\frac{1}{2p-1}}dx  \notag \\
&\leq &C\int_{\mathbb{R}}\left\Vert \delta \left( u_{t}\right) \right\Vert
_{4\left( 2p-1\right) }^{3}\int_{y_{1}}^{y_{2}}\exp \left( \frac{-\left(
z-x\right) ^{2}}{32\left( 2p-1\right) c\left( t-r\right) }\right) dzdx
\notag \\
&\leq &C\left( t-r\right) ^{-1}\left( y_{2}-y_{1}\right) .  \label{space5}
\end{eqnarray}%
On the other hand, the left hand side of (\ref{space5}) can be estimated
differently again by using Lemma \ref{dens bd}:
\begin{eqnarray}
&&\int_{\mathbb{R}}\left( E\left\vert p^{W}\left( r,x;t,y_{1}\right)
-p^{W}\left( r,x;t,y_{2}\right) \right\vert ^{2\left( 2p-1\right) }\right) ^{%
\frac{1}{2p-1}}dx  \notag \\
&\leq &2\int_{\mathbb{R}}\Sigma _{_{i=1,2}}\left( E\left\vert p^{W}\left(
r,x;t,y_{i}\right) \right\vert ^{2\left( 2p-1\right) }\right) ^{\frac{1}{2p-1%
}}dx  \notag \\
&\leq &C_{p}\int_{\mathbb{R}}\Sigma _{_{i=1,2}}\left\Vert \delta \left(
u_{t}\right) \right\Vert _{4\left( 2p-1\right) }^{2}\exp \left( \frac{%
-\left( y_{i}-x\right) ^{2}}{64pc\left( t-r\right) }\right) dx\leq C\left(
t-r\right) ^{-\frac{1}{2}}.~~  \label{space6}
\end{eqnarray}%
Then (\ref{space5}) and (\ref{space6}) yield that for any $\alpha ,\beta >0$
with $\alpha +\beta =1$%
\begin{equation}
\int_{\mathbb{R}}\left( E\left\vert p^{W}\left( r,x;t,y_{1}\right)
-p^{W}\left( r,x;t,y_{2}\right) \right\vert ^{2\left( 2p-1\right) }\right) ^{%
\frac{1}{2p-1}}dx\leq C\left( t-r\right) ^{-\alpha -\frac{1}{2}\beta }\left(
y_{2}-y_{1}\right) ^{\alpha }.  \label{space7}
\end{equation}%
Since $\mu $ is bounded, it follows from \cite[Lemma 4.1]{LWXZ11} that
\begin{eqnarray}
&&E\left\vert \int_{0}^{t}\int_{\mathbb{R}}\left( p^{W}\left(
r,x;t,y_{2}\right) -p^{W}\left( r,x;s,y_{1}\right) \right) ^{2}Z\left(
drdx\right) \right\vert ^{2p}  \notag \\
&\leq &C\left( E\left\vert \int_{0}^{t}\int_{\mathbb{R}}\left( p^{W}\left(
r,x;t,y_{2}\right) -p^{W}\left( r,x;s,y_{1}\right) \right)
^{2}drdx\right\vert ^{2p-1}\right) ^{\frac{p}{2p-1}},~~  \label{BDG1}
\end{eqnarray}%
for any $p\geq 1$, $0\leq s\leq t\leq T$ and $y_{1},y_{2}\in \mathbb{R}$.
Then, applying Minkowski's inequality we obtain for any $0<\alpha \,<1$,
\begin{eqnarray*}
&&\left( E\left\vert X_{t,2}\left( y_{2}\right) -X_{t,2}\left( y_{1}\right)
\right\vert ^{2p}\right) ^{\frac{1}{p}} \\
&\leq &\int_{0}^{t}\int_{\mathbb{R}}\left( E\left\vert p^{W}\left(
r,x;t,y_{1}\right) -p^{W}\left( r,x;t,y_{2}\right) \right\vert ^{2\left(
2p-1\right) }\right) ^{\frac{1}{2p-1}}dxdr \\
&\leq &C\int_{0}^{t}\left( t-r\right) ^{-\alpha -\frac{1}{2}\beta }\left(
y_{2}-y_{1}\right) ^{\alpha }dr\leq C\left( y_{2}-y_{1}\right) ^{\alpha }
\end{eqnarray*}%
since $\left( t-r\right) ^{-\alpha -\frac{1}{2}\beta }=\left( t-r\right)
^{-(1+\alpha )/2}$ is integrable for all $0<{\alpha }<1$.

Now we consider $X_{t,1}(y)$ in (\ref{cnvl rep}). Applying Minkowski's
inequality and using (\ref{space4}) with $2p-1$ replaced by $p$ we get%
\begin{eqnarray*}
&&E\left\vert X_{t,1}\left( y_{2}\right) -X_{t,1}\left( y_{1}\right)
\right\vert ^{2p} \\
&\leq &\left( \int_{\mathbb{R}}\left( E\left\vert p\left( 0,x;t,y_{1}\right)
-p^{W}\left( 0,x;t,y_{2}\right) \right\vert ^{2p}\right) ^{\frac{1}{2p}}\mu
\left( x\right) dx\right) ^{2p} \\
&\leq &C\left\{ \int_{\mathbb{R}}\left( \int_{y_{1}}^{y_{2}}\left\Vert
p^{W}\left( 0,x;t,z\right) \right\Vert _{2p}dz\right) ^{1/2}\left\Vert
\delta \left( u_{t}\right) \right\Vert _{4p}\mu \left( x\right) dx\right\}
^{2p} \\
&\leq &C\left\Vert \delta \left( u_{t}\right) \right\Vert
_{4p}^{2p}\left\Vert \mu \right\Vert _{L^{2}(\mathbb{R})}^{2p}\left( \int_{%
\mathbb{R}}\int_{y_{1}}^{y_{2}}\exp \left( -\frac{\left( z-x\right) ^{2}}{%
64pct}\right) dzdx\right) ^{p} \\
&\leq &C\left\Vert \mu \right\Vert _{L^{2}(\mathbb{R})}^{2p}t^{-p}\left(
y_{2}-y_{1}\right) ^{p}.
\end{eqnarray*}

This completes the proof.
\end{proof}

\setcounter{equation}{0}

\section{H\"older continuity in time variable}

In this section we show that for any fixed $y\in \mathbb{R}$, $X_{t}(y)$ is H%
\"{o}lder continuous in $t$ with any exponent in $(0,1/4)$.

\begin{prop}
\label{time}Suppose that $h\in H_{2}^{2}(\mathbb{R})$ and $X_{0}$ has a
bounded density $\mu \in L^{2}(\mathbb{R})$. Then, for any $p\geq 1$, $0\leq
s<t\leq T$ and $y\in \mathbb{R}$,%
\begin{equation*}
E\left\vert X_{t}\left( y\right) -X_{s}\left( y\right) \right\vert ^{2p}\leq
C(1+t^{-p})\left( t-s\right) ^{\frac{p}{2}-\frac{1}{4}},
\end{equation*}%
where the constant $C$ depending only on $p,T$, $\left\Vert h\right\Vert
_{2,2}$ and $\left\Vert \mu \right\Vert _{L^{2}(\mathbb{R})}$.
\end{prop}

We need some preparations to prove the above result.

Suppose $0<s<t$. We start by estimating $X_{\cdot ,2}\left( y\right) $ in (%
\ref{cnvl rep}) and we write%
\begin{eqnarray}
&&X_{t,2}\left( y\right) -X_{s,2}\left( y\right) =\int_{0}^{s}\int_{\mathbb{R%
}}\left( p^{W}\left( r,x;t,y\right) -p^{W}\left( r,x;s,y\right) \right)
Z\left( drdx\right)  \notag \\
&&+\int_{s}^{t}\int_{\mathbb{R}}p^{W}\left( r,x;t,y\right) Z\left(
drdx\right) .  \label{time 0}
\end{eqnarray}%
We are going to estimate the two terms separately.

\begin{lemma}
For any $0\leq s<t\leq T$, $y\in \mathbb{R}$ and $p\geq 1$, we have%
\begin{equation}
E\left( \int_{0}^{s}\int_{\mathbb{R}}\left( p^{W}\left( r,x;t,y\right)
-p^{W}\left( r,x;s,y\right) \right) Z\left( drdx\right) \right) ^{2p}\leq
C\left( t-s\right) ^{\frac{p}{2}-\frac{1}{4}}.  \label{time1}
\end{equation}
\end{lemma}

\begin{proof}
From (\ref{dens}), we have for $0<r<s<t\leq T$,
\begin{eqnarray*}
&&p^{W}\left( r,x;t,y\right) -p^{W}\left( r,x;s,y\right) =E^{B}\left[
\mathbf{1}_{\left\{ \xi _{t}>y\right\} }\delta \left( u_{t}\right) -\mathbf{1%
}_{\left\{ \xi _{s}>y\right\} }\delta \left( u_{s}\right) \right] \\
&=&E^{B}\left[ \left( 1_{\left\{ \xi _{t}>y\right\} }-\mathbf{1}_{\left\{
\xi _{s}>y\right\} }\right) \delta \left( u_{t}\right) +\mathbf{1}_{\left\{
\xi _{s}>y\right\} }\delta \left( u_{t}-u_{s}\right) \right] ,
\end{eqnarray*}%
Let $I_{1}\equiv \left( 1_{\left\{ \xi _{t}>y\right\} }-\mathbf{1}_{\left\{
\xi _{s}>y\right\} }\right) \delta \left( u_{t}\right) $ and $I_{2}\equiv
\mathbf{1}_{\left\{ \xi _{s}>y\right\} }\delta \left( u_{t}-u_{s}\right) $.
Then (\ref{BDG1}) implies%
\begin{eqnarray}
&&E\left( \int_{0}^{s}\int_{\mathbb{R}}\left( p^{W}\left( r,x;t,y\right)
-p^{W}\left( r,x;s,y\right) \right) Z\left( drdx\right) \right) ^{2p}  \notag
\\
&\leq &\left[ E\left( \int_{0}^{s}\int_{\mathbb{R}}\left(
E^{B}[I_{1}+I_{2}]\right) ^{2}drdx\right) ^{2p-1}\right] ^{\frac{p}{2p-1}}
\notag \\
&\leq &C\sum_{i=1,2}\left[ E\left( \int_{0}^{s}\int_{\mathbb{R}}\left(
E^{B}I_{i}\right) ^{2}drdx\right) ^{2p-1}\right] ^{\frac{p}{2p-1}}.
\label{timeA0}
\end{eqnarray}

First, we study the term $I_{1}$. Note that
\begin{equation*}
\left( 1_{\left\{ \xi _{t}>y\right\} }-\mathbf{1}_{\left\{ \xi
_{s}>y\right\} }\right) ^{2}=1_{\left\{ \xi _{s}\leq y<\xi _{t}\right\} }+%
\mathbf{1}_{\left\{ \xi _{t}\leq y<\xi _{s}\right\} }=:A_{1}+A_{2}.
\end{equation*}%
Then we can write%
\begin{eqnarray}
&&\left[ E\left( \int_{0}^{s}\int_{\mathbb{R}}E^{B}I_{1}^{2}drdx\right)
^{2p-1}\right] ^{\frac{1}{2p-1}}  \notag \\
&=&\left[ E\left( \int_{0}^{s}\int_{\mathbb{R}}E^{B}\left[ \left(
A_{1}+A_{2}\right) \delta \left( u_{t}\right) \right] ^{2}drdx\right) ^{2p-1}%
\right] ^{\frac{1}{2p-1}}  \notag \\
&\leq &2\sum_{i=1,2}\left[ E\left( \int_{0}^{s}\int_{\mathbb{R}}E^{B}\left[
A_{i}\delta \left( u_{t}\right) \right] ^{2}drdx\right) ^{2p-1}\right] ^{%
\frac{1}{2p-1}}.  \label{timeA-1}
\end{eqnarray}%
Applying Minkowski, Jensen and H\"{o}lder's inequalities we deduce that for $%
i=1,2$ and for any conjugate pair $\left( p_{1},q_{1}\right) $%
\begin{eqnarray}
&&\left[ E\left( \int_{0}^{s}\int_{\mathbb{R}}E^{B}\left[ A_{i}\delta \left(
u_{t}\right) \right] ^{2}drdx\right) ^{2p-1}\right] ^{\frac{1}{2p-1}}  \notag
\\
&\leq &\int_{0}^{s}\left( E\left\vert \int_{\mathbb{R}}A_{i}\left\vert
\delta \left( u_{t}\right) \right\vert ^{2}dx\right\vert ^{2p-1}\right) ^{%
\frac{1}{2p-1}}dr  \notag \\
&\leq &\int_{0}^{s}\left\Vert \left( \int_{\mathbb{R}}A_{i}dx\right) ^{\frac{%
1}{p_{1}}}\left( \int_{\mathbb{R}}A_{i}\left\vert \delta \left( u_{t}\right)
\right\vert ^{2q_{1}}dx\right) ^{\frac{1}{q_{1}}}\right\Vert _{2p-1}dr
\notag \\
&\leq &\int_{0}^{s}\left\Vert \left\vert \int_{\mathbb{R}}A_{i}dx\right\vert
^{\frac{1}{p_{1}}}\right\Vert _{2\left( 2p-1\right) }\left\Vert \left( \int_{%
\mathbb{R}}A_{i}\left\vert \delta \left( u_{t}\right) \right\vert
^{2q_{1}}dx\right) ^{\frac{1}{q_{1}}}\right\Vert _{2\left( 2p-1\right) }dr.
\label{timeA-2}
\end{eqnarray}%
Notice that
\begin{eqnarray*}
\left\{ \xi _{s}\leq y<\xi _{t}\right\} &=&\left\{
y-B_{r}^{t}-I_{r}^{t}(h)<x\leq y-B_{r}^{s}-I_{r}^{s}(h)\right\} , \\
\left\{ \xi _{t}\leq y<\xi _{s}\right\} &=&\left\{
y-B_{r}^{s}-I_{r}^{s}(h)<x\leq y-B_{r}^{t}-I_{r}^{t}(h)\right\} .
\end{eqnarray*}%
Then, for $i=1,2$, we have
\begin{equation*}
\left\vert \int_{\mathbb{R}}A_{i}dx\right\vert =\left\vert
B_{s}^{t}+I_{s}^{t}(h)\right\vert .
\end{equation*}%
Hence for $p_{1}=1-\frac{1}{2p}$,
\begin{equation}
\left\Vert \left\vert \int_{\mathbb{R}}A_{i}dx\right\vert ^{\frac{1}{p_{1}}%
}\right\Vert _{2\left( 2p-1\right) }\leq C\left( t-s\right) ^{\frac{1}{2}-%
\frac{1}{4p}}.  \label{timeA-21}
\end{equation}%
On the other hand, we have
\begin{eqnarray*}
\left\{ \xi _{s}\leq y<\xi _{t}\right\} &=&\left\{
B_{r}^{s}+I_{r}^{s}(h)\leq y-x<B_{r}^{t}+I_{r}^{t}(h)\right\} \\
&\subset &\left\{ \left\vert x-y\right\vert \leq \left\vert
B_{r}^{t}+I_{r}^{t}(h)\right\vert +\left\vert
B_{r}^{s}+I_{r}^{s}(h)\right\vert \right\} .
\end{eqnarray*}%
Similarly
\begin{equation*}
\left\{ \xi _{t}\leq y<\xi _{s}\right\} \subset \left\{ \left\vert
x-y\right\vert \leq \left\vert B_{r}^{t}+I_{r}^{t}(h)\right\vert +\left\vert
B_{r}^{s}+I_{r}^{s}(h)\right\vert \right\} .
\end{equation*}%
Applying Chebyshev's inequality and (\ref{Gaussian}), we deduce that for $%
i=1,2$,
\begin{eqnarray}
&&E\left( A_{i}\right) \leq EP^{B}\left\{ \left\vert x-y\right\vert \leq
\left\vert B_{r}^{t}+I_{r}^{t}(h)\right\vert +\left\vert
B_{r}^{s}+I_{r}^{s}(h)\right\vert \right\}  \notag \\
&\leq &\exp \left( \frac{-\left( x-y\right) ^{2}}{32c\left( t-r\right) }%
\right) E\exp \left( \frac{\left\vert B_{r}^{t}+I_{r}^{t}(h)\right\vert ^{2}%
}{16c(t-r)}+\frac{\left\vert B_{r}^{s}+I_{r}^{s}(h)\right\vert ^{2}}{16c(s-r)%
}\right)  \notag \\
&\leq &2\exp \left( -\frac{\left( x-y\right) ^{2}}{32c\left( t-r\right) }%
\right) .  \label{timeA-22}
\end{eqnarray}%
Using Minkowski and H\"{o}lder's inequalities, from (\ref{timeA-22}) and
Lemma \ref{Edelta u} we obtain that for $q_{1}=2p\leq 2\left( 2p-1\right) $,
\begin{eqnarray}
&&\left\Vert \left( \int_{\mathbb{R}}A_{i}\left\vert \delta \left(
u_{t}\right) \right\vert ^{2q_{1}}dx\right) ^{\frac{1}{q_{1}}}\right\Vert
_{2\left( 2p-1\right) }\leq \left( \int_{\mathbb{R}}\left\Vert
A_{i}\left\vert \delta \left( u_{t}\right) \right\vert ^{2q_{1}}\right\Vert
_{\frac{2\left( 2p-1\right) }{q_{1}}}dx\right) ^{\frac{1}{q_{1}}}  \notag \\
&\leq &\left( \int_{\mathbb{R}}\left( EA_{i}\right) ^{\frac{q_{1}}{4\left(
2p-1\right) }}\left\Vert \delta \left( u_{t}\right) \right\Vert _{8\left(
2p-1\right) }^{2q_{1}}dx\right) ^{\frac{1}{q_{1}}}\leq C\left( t-r\right) ^{%
\frac{1}{4p}-1}.  \label{timeA-23}
\end{eqnarray}%
Substituting (\ref{timeA-21}) and (\ref{timeA-23}) into (\ref{timeA-2}) we
obtain%
\begin{eqnarray}
&&\left[ E\left( \int_{0}^{s}\int_{\mathbb{R}}E^{B}\left[ A_{i}\delta \left(
u_{t}\right) \right] ^{2}drdx\right) ^{2p-1}\right] ^{\frac{1}{2p-1}}  \notag
\\
&\leq &C\left( t-s\right) ^{\frac{1}{2}-\frac{1}{4p}}\int_{0}^{s}\left(
t-r\right) ^{\frac{1}{4p}-1}dr\leq C\left( t-s\right) ^{\frac{1}{2}-\frac{1}{%
4p}}.  \label{timeA-24}
\end{eqnarray}%
Combining (\ref{timeA-1}) and (\ref{timeA-24}), we have%
\begin{equation}
\left[ E\left( \int_{0}^{s}\int_{\mathbb{R}}E^{B}I_{1}^{2}drdx\right) ^{2p-1}%
\right] ^{\frac{1}{2p-1}}\leq C\left( t-s\right) ^{\frac{1}{2}-\frac{1}{4p}}.
\label{timeI1}
\end{equation}

We turn into the term $I_{2}$. From Lemma \ref{dens ineq} we can deduce as
in Lemma \ref{dens bd} that%
\begin{equation*}
\left( E\left( E^{B}I_{2}\right) ^{2\left( 2p-1\right) }\right) ^{\frac{1}{%
2p-1}}\leq 2\exp \left( \frac{-\left( x-y\right) ^{2}}{32\left( 2p-1\right)
c\left( s-r\right) }\right) \left\Vert \delta \left( u_{t}-u_{s}\right)
\right\Vert _{4\left( 2p-1\right) }^{2}.
\end{equation*}%
Then applying Minkowski's inequality and Lemma \ref{delta timediff}, we
obtain
\begin{eqnarray}
&&\left[ E\left( \int_{0}^{s}\int_{\mathbb{R}}\left( E^{B}I_{2}\right)
^{2}drdx\right) ^{2p-1}\right] ^{\frac{1}{2p-1}}\leq \int_{0}^{s}\int_{%
\mathbb{R}}\left( E\left( E^{B}I_{2}\right) ^{2\left( 2p-1\right) }\right) ^{%
\frac{1}{2p-1}}drdx  \notag \\
&\leq &2\int_{0}^{s}\int_{\mathbb{R}}\exp \left( -\frac{\left( x-y\right)
^{2}}{32\left( 2p-1\right) \left( s-r\right) }\right) \left\Vert \delta
\left( u_{t}-u_{s}\right) \right\Vert _{4\left( 2p-1\right) }^{2}dxdr  \notag
\\
&\leq &C\left( t-s\right) \int_{0}^{s}\left( s-r\right) ^{\frac{1}{2}%
-1}\left( t-r\right) ^{-1}dr\leq C\left( t-s\right) ^{\frac{1}{2}-\frac{1}{4p%
}},  \label{timeI2}
\end{eqnarray}%
where in the last step we used that $\left( t-r\right) ^{-1}\leq \left(
t-s\right) ^{-\frac{1}{2}-\varepsilon }\left( s-r\right) ^{-\frac{1}{2}%
+\varepsilon }$ for any $\varepsilon >0$.

Substituting (\ref{timeI1}) and (\ref{timeI2}) in (\ref{timeA0}) we obtain (%
\ref{time1}).
\end{proof}

\begin{lemma}
For any $0\leq s<t\leq T$ and any $y\in \mathbb{R}$ and $p\geq 1$, we have%
\begin{equation}
E\left( \int_{s}^{t}\int_{\mathbb{R}}p^{W}\left( r,x;t,y\right) Z\left(
drdx\right) \right) ^{2p}\leq C\left( t-s\right) ^{\frac{p}{2}}.
\label{time2}
\end{equation}
\end{lemma}

\begin{proof}
Since $\mu $ is bounded, it follows from \cite[Lemma 4.1]{LWXZ11} that
\begin{equation}
E\left\vert \int_{s}^{t}\int_{\mathbb{R}}p^{W}\left( r,x;t,y\right)
^{2}Z\left( drdx\right) \right\vert ^{2p}\leq C\left( E\left\vert
\int_{s}^{t}\int_{\mathbb{R}}p^{W}\left( r,x;t,y\right) ^{2}drdx\right\vert
^{2p-1}\right) ^{\frac{p}{2p-1}},  \label{BDG2}
\end{equation}%
for any $p\geq 1$ and $y\in \mathbb{R}$. Applying Minkowski's inequality,
Lemma \ref{dens bd}\ and Lemma \ref{Edelta u} we obtain%
\begin{eqnarray*}
&&\left[ E\left( \int_{0}^{s}\int_{\mathbb{R}}\left\vert p^{W}\left(
r,x;t,y\right) \right\vert ^{2}drdx\right) ^{2p-1}\right] ^{\frac{1}{2p-1}}
\\
&\leq &C\int_{s}^{t}\int_{\mathbb{R}}\left( E\left\vert p^{W}\left(
r,x;t,y\right) \right\vert ^{2\left( 2p-1\right) }\right) ^{\frac{1}{2p-1}%
}drdx \\
&\leq &C\int_{s}^{t}\int_{\mathbb{R}}\exp \left( -\frac{\left( x-y\right)
^{2}}{32c\left( t-r\right) }\right) \left\Vert \delta \left( u_{t}\right)
\right\Vert _{4\left( 2p-1\right) }^{2}drdx \\
&\leq &C\int_{s}^{t}\left( t-r\right) ^{\frac{1}{2}-1}dr\leq C\left(
t-s\right) ^{\frac{1}{2}}.
\end{eqnarray*}%
Then (\ref{time2}) follows immediately.
\end{proof}

In summary of the above two lemmas, we get

\begin{proposition}
\label{timeX2prop} For any $p\geq 1$, $0\leq s<t\leq T$ and $y\in \mathbb{R}$%
, we have%
\begin{equation*}
E\left\vert X_{t,2}\left( y\right) -X_{s,2}\left( y\right) \right\vert
^{2p}\leq C\left( t-s\right) ^{\frac{p}{2}-\frac{1}{4}}.
\end{equation*}
\end{proposition}

Now we consider $X_{t,1}\left( y\right) $. Note that
\begin{eqnarray*}
&&E\left\vert X_{t,1}\left( y\right) -X_{s,1}\left( y\right) \right\vert
^{2p}=E\left\vert \int_{\mathbb{R}}\left( p^{W}\left( 0,z;t,y\right)
-p^{W}\left( 0,z;t,y\right) \right) \mu (z)dz\right\vert ^{2p} \\
&=&E\left\vert \int_{\mathbb{R}}\left( E^{B}\left[ \mathbf{1}_{\left\{ \xi
_{t}>y\right\} }\delta \left( u_{t}\right) -\mathbf{1}_{\left\{ \xi
_{s}>y\right\} }\delta \left( u_{s}\right) \right] \right) \mu
(z)dz\right\vert ^{2p}.
\end{eqnarray*}%
Then, similar to the proof for $X_{\cdot ,2}\left( y\right) $ we get
estimates for $X_{\cdot ,1}\left( y\right) .$

\begin{proposition}
\label{timeX1prop} For any $p\geq 1,$ $0\leq s<t\leq T$ and any $y\in
\mathbb{R}$, we have%
\begin{equation}
E\left\vert X_{t,1}\left( y\right) -X_{s,1}\left( y\right) \right\vert
^{2p}\leq C\left( 1+t^{-p}\right) \left( t-s\right) ^{\frac{1}{2}p}.
\label{time X1}
\end{equation}
\end{proposition}

\begin{proof}
Let $I_{1}\equiv \left( 1_{\left\{ \xi _{t}>y\right\} }-\mathbf{1}_{\left\{
\xi _{s}>y\right\} }\right) \delta \left( u_{t}\right) $ and $I_{2}\equiv
\mathbf{1}_{\left\{ \xi _{s}>y\right\} }\delta \left( u_{t}-u_{s}\right) $.
Then,%
\begin{equation*}
E\left\vert X_{t,1}\left( y\right) -X_{s,1}\left( y\right) \right\vert
^{2p}=E\left\vert \int_{\mathbb{R}}\mu (x)E^{B}[I_{1}+I_{2}]dx\right\vert
^{2p}.
\end{equation*}

Noticing that $\left\vert 1_{\left\{ \xi _{t}>y\right\} }-\mathbf{1}%
_{\left\{ \xi _{s}>y\right\} }\right\vert =1_{\left\{ \xi _{s}\leq y<\xi
_{t}\right\} }+\mathbf{1}_{\left\{ \xi _{t}\leq y<\xi _{s}\right\}
}=:A_{1}+A_{2},$ and applying Fubini's theorem, Jensen, H\"{o}lder and
Minkowski's inequalities, we obtain%
\begin{eqnarray*}
&&E\left\vert \int_{\mathbb{R}}\mu (x)E^{B}\left\vert I_{1}\right\vert
dx\right\vert ^{2p}\leq \sum_{i=1,2}E\left\vert \int_{\mathbb{R}}\mu (x)E^{B}%
\left[ A_{i}\delta \left( u_{t}\right) \right] dx\right\vert ^{2p} \\
&\leq &\sum_{i=1,2}E[\left( \int_{\mathbb{R}}\left\vert \mu (x)\delta \left(
u_{t}\right) \right\vert ^{2}dx\right) ^{p}\left\vert \int_{\mathbb{R}%
}A_{i}dx\right\vert ^{p}] \\
&\leq &\sum_{i=1,2}\left( \int_{\mathbb{R}}\left\vert \mu (x)\right\vert
^{2}\left\Vert \delta \left( u_{t}\right) \right\Vert _{4p}^{2}dx\right)
^{p}\left( E\left\vert B_{s}^{t}+I_{s}^{t}(h)\right\vert ^{2p}\right) ^{%
\frac{1}{2}} \\
&\leq &C\left( 1+\left\Vert h\right\Vert ^{2}\right) \left\Vert \mu
\right\Vert _{L^{2}}^{2p}t^{-p}\left( t-s\right) ^{\frac{1}{2}p}.
\end{eqnarray*}

For the term $I_{2},$ using Minkowski's inequality, $\left( \ref{dens ctrl}%
\right) $ and (\ref{tdelta}) \ with $r=0$ we have
\begin{eqnarray*}
&&E\left\vert \int_{\mathbb{R}}\mu (x)E^{B}\left\vert I_{2}\right\vert
dx\right\vert ^{2p}=\left( \int_{\mathbb{R}}\left\vert \mu (x)\right\vert
\left( E\left\vert E^{B}\mathbf{1}_{\left\{ \xi _{s}>y\right\} }\delta
\left( u_{t}-u_{s}\right) \right\vert ^{2p}\right) ^{\frac{1}{2p}}dx\right)
^{2p} \\
&\leq &C\left\Vert \mu \right\Vert _{\infty }^{2p}\left( \int_{\mathbb{R}%
}\exp \left( -\frac{\left( x-y\right) ^{2}}{32cs}\right) \left\Vert \delta
\left( u_{t}-u_{s}\right) \right\Vert _{4p}dx\right) ^{2p} \\
&\leq &C\left\Vert \mu \right\Vert _{\infty }^{2p}t^{-p}\left( t-s\right)
^{p}.
\end{eqnarray*}%
Then we can conclude (\ref{time X1}).
\end{proof}

\begin{proof}[Proof of Proposition \protect\ref{time}]
It follows from Proposition \ref{timeX2prop} and Proposition \ref{timeX1prop}
.
\end{proof}

\begin{proof}[Proof of Theorem \protect\ref{main}]
It follows from Proposition \ref{Space} and Proposition \ref{time}.
\end{proof}

\smallskip \noindent \textbf{Acknowledgment} The authors would like to thank
Jie Xiong for bringing us of this topic.

\end{document}